\documentclass{article}

\usepackage{verbatim}
\usepackage{graphics}
\usepackage{graphicx}
\usepackage{amsmath, amssymb, amsfonts, amsthm}
\usepackage{url}
\usepackage{setspace}
\newtheorem{thm}{Theorem}[section]
\newtheorem{cor}[thm]{Corollary}
\newtheorem{lem}[thm]{Lemma}

\theoremstyle{definition}
\newtheorem{defn}[thm]{Definition}

\begin{document}
\author{George B. Purdy\footnote{george.purdy@uc.edu. Department Of Computer Science, 814 Rhodes Hall, Cincinnati, OH 45221}\\
Justin W. Smith\footnote{smith5jw@email.uc.edu. Department Of Computer Science, 814 Rhodes Hall, Cincinnati, OH 45221} }

\title{Lines, Circles, Planes and Spheres}
\maketitle

\begin{abstract}
Let $S$ be a set of $n$ points in $\mathbb{R}^3$, no three collinear and not all coplanar. 
If at most $n-k$ are coplanar and $n$ is sufficiently large, 
the total number of planes determined is at least $1 + k \binom{n-k}{2}-\binom{k}{2}\left(\frac{n-k}{2}\right)$.
For similar conditions and sufficiently large $n$, (inspired by the work of P. D. T. A. Elliott in \cite{Ell67})
we also show that the number of spheres determined by $n$ points is at least
$1+\binom{n-1}{3}-t_3^{orchard}(n-1)$, and this bound is best possible under its hypothesis. 
(By $t_3^{orchard}(n)$, we are denoting the maximum number of three-point lines 
attainable by a configuration of $n$ points, no four collinear, in the plane, i.e., the classic Orchard Problem.)
New lower bounds are also given for both lines and circles.
\end{abstract}

\section{Introduction}

The problem of maximizing the number of three-point lines goes back as far as 1821 when several such problems appeared in \cite{JJ21} 
in a section called ``Trees planted in rows.'' 
The first question of the section essentially asks (but in a rhyming verse) how to plant nine trees such that they form ten rows of three.
That section also includes several variations on this question.
Many years later, Sylvester demonstrated a configuration of $n$ points in the plane, lying on a cubic curve,
that determines $\sim \frac{1}{8}n^2$ three-point lines \cite[pp. 106-107]{Mil68}.
(More specifically, Sylvester's configuration was answering his own question of how to plant 81 trees to form 800 rows of three!)
See \cite{BGS74} and \cite[pp. 315--318]{BMP05} for a complete history of the Orchard Problem.

In Section \ref{sec-orchard}, we discuss an unexpected equivalence to the Orchard Problem that we found while 
researching a lower bound for the number of spheres determined by a set of points.
This lower bound, best possible under its hypothesis, is $1+\binom{n-1}{3}-t_3^{orchard}(n-1)$,
where $t_3^{orchard}(n)$ is the upper bound on the number of three-point lines determined by $n$ points, no four collinear, in the plane.
This equivalence was discovered after realizing that a similar subtractive term exists for circles (but was overlooked in \cite{Ell67}).
We must first, however, present several lower bounds on determined planes which, by using methods inspired by P. D. T. A. Elliott, provide
the tools necessary to obtain our results for spheres.

Motzkin's 1951 paper, \cite{Mot51}, is often cited for his conjecture (also conjectured by Dirac in \cite{Dir51})
that among any set of $2n$ noncollinear points in the plane
there exist $n$ ordinary (i.e., two-point) lines. Included in that paper, \cite{Mot51}, are several results concerning planes (or hyperplanes) 
determined by points in three, or higher, dimensional space. 

Using two counter-examples, i.e., the Desargues configuration and a set of points all lying on two skew lines in 
$\mathbb{R}^3$, Motzkin demonstrated that a three-point plane need not exist among a finite set of 
noncoplanar points.
So Motzkin defined the proper analog to an ``ordinary line'' to be an ``ordinary plane'', 
i.e., a plane in which all but one of the incident points may lie on a line. 
Similarly, an ``ordinary hyperplane'' in $d$-space is defined to be one in which all but one of its points
are incident to a ($d-2$)-flat (i.e., an affine subspace of dimension $d-2$).
S. Hansen later proved that an ordinary hyperplane necessarily exists among a finite set of points in ($d\geqslant3$)-space \cite{Han65}.
Hansen also found a new infinite family of configurations in $\mathbb{R}^3$ with no three-point planes \cite{Han80}.

Given a set of $n$ points in $\mathbb{R}^3$, let $m$ be the number of planes determined, and let $t$ be the number of lines. Purdy conjectures, 
under various conditions, that $m-t+n\geqslant2$ and $m\geqslant t$.  Extending this to $\mathbb{R}^d$, it is conjectured by Purdy that 
under certain conditions the number of hyperplanes determined is at least the number of ($d-2$)-flats.  (See Section \ref{sec:cor-conj} for discussion of these conjectures.)

Compared to the work done concerning points in the plane, 
relatively little has been done to extend Motzkin's results concerning planes in $\mathbb{R}^3$.
Several results in this area were obtained by
Erd\H{o}s and Purdy in \cite{EP77}, 
with which the present paper is related. In \cite{EP77},
they prove that given a finite set of $n\geqslant552$ points, not all coplanar and no three collinear,
there are at least $\binom{n-1}{2}+1$ planes determined.
(This result verified $m-t+n\geqslant2$ for that case.)
They also proved that given such a point set there exist $\frac{1}{2}n^2-cn$ 
determined planes (for a suitable $c$) incident to at most four points.
The present article improves upon both of these results.

\section{Lines Determined by Points in $\mathbb{R}^2$}

This section contains a new result for lines in the plane.
Its inclusion was motivated, in part, because it demonstrates in a more familiar setting
the method of proof that will be used again for planes in 3-space.

In this section, we will consider a set of $n$ points in $\mathbb{R}^2$.
Let $t_k$ be the number of lines incident to exactly $k$ of these points.
Let $t$ be the total number of determined lines, i.e., $t=\sum_{i\geqslant2} t_i$

In \cite{EP78}, Erd\H{o}s and Purdy prove the following lemma.
For the convenience of the reader, we provide a proof.
\begin{lem}
\label{lem:lb-ol}
If $r$ points are on a line, $l$, and $s$ points are not on $l$,
then 
\[
t_2 \geqslant rs-s(s-1).
\]
\end{lem}
\begin{proof}
The result is true for $s=0$ and $s=1$, so suppose $s>1$. 
We shall use induction on $s$.
Let $p$, not on $l$, be the last point to be added for a total of $s$ points off of $l$.
Assume it's true for $s-1$.
By adding point $p$, at most $s-1$ existing lines are spoiled, 
and $r$ lines are created of which at most 
$s-1$ already exist.
Thus,
\begin{multline*}
t_2 \geqslant r(s-1)-(s-1)(s-2)-(s-1)+r-(s-1)=rs-s(s-1).
\end{multline*}
\end{proof}

The following lemma is due to Elliot \cite{Ell67}.
\begin{lem}
\label{lem:lb-23-pt-lines}
The number of lines determined by at most three points in a plane is at least one-half the total number of lines.
More specifically,
\[
t_2 + t_3 \geqslant\frac{t+3}{2}.
\]
\end{lem}
\begin{proof}
Beginning with Melchior's Inequality \cite{Mel41},
\[
t_2 \geqslant 3 + \sum_{k \geqslant 3} (k-3) t_k.
\]
Adding $t_2+t_3$ to both sides yields,
\[
2t_2+t_3\geqslant3+t_2+t_3+t_4+2t_5+\ldots.
\]
Obviously, $t=t_2+t_3+t_4+...$. So,
\[
2(t_2+t_3)\geqslant3+t.
\]
The lemma follows.
\end{proof}

We shall use one more lemma, due to Kelly and Moser \cite{KM58}, which is also a consequence of Melchior's Inequality.
Let $r_i$ be the total number of points incident to exactly $i$ determined lines.
\begin{lem}
\label{lem:pt-line-incidences}
The total number of lines is at least one-third the total number of point-line incidences.  That is,
\[
3t-3\geqslant \sum_{i\geqslant2} i\cdot t_i = \sum_{i\geqslant2} i\cdot r_i.
\]
\end{lem}
\begin{proof}
Again starting from Melchior's Inequality,
\[
-3\geqslant \sum_{k\geqslant2} (k-3)t_k = -t_2 + t_4+2t_5+3t_6+\ldots.
\]
By adding $3t$ to both sides,
\[
3t \geqslant 3+2t_2+3t_3+4t_4+5t_5+6t_6+\ldots.
\]
The lemma follows.
\end{proof}

By combining Lemmas  \ref{lem:lb-23-pt-lines} and \ref{lem:pt-line-incidences} we see the following,
\begin{equation*}
6(t_2+t_3)\geqslant3(t+3)\geqslant12+\sum_{i\geqslant2} i\cdot t_i = 12+\sum_{i\geqslant2} i\cdot r_i,
\end{equation*}
or simplified,
\begin{equation}
\label{eqn:lb-t23-incidences}
t_2+t_3\geqslant2+\frac{1}{6}\sum_{i\geqslant2} i\cdot r_i.
\end{equation}

\begin{thm}
\label{thm:2-3-pt-lines}
Let $S$ be a set of $n$ points in $\mathbb{R}^2$. If $n\geqslant 72k^2+2k-1$, and no more than $n-k$ points
are collinear, then
\[
t_2+t_3\geqslant k(n-k)-k(k-1).
\] 
\end{thm}

We define the degree of a point to be the number of determined lines to which it is incident.
\begin{proof}
\item{Case 1: There exist two points, $p$ and $q$, of degree $< 6k$.}

Let $l$ be the line determined by $p$ and $q$.  Assume $l$ is incident to exactly $n-x$ points.
Each line through $p$, other than $l$, must be incident to less than  $6k$ points, 
otherwise the degree of point $q$ would be too high, i.e. a contradiction.  Likewise for lines through $q$.
Thus, $k \leqslant x < 36k^2 =(6k)^2$.

By Lemma \ref{lem:lb-ol}, we know $t_2 \geqslant f(x)\buildrel\rm def\over= x(n-x)-x(x-1) = -2x^2-(n+1)x$.
The second derivative of $f(x)$ is negative,  i.e., $f''(x) = -4$.  Therefore,
\[
\min_{k\leqslant x\leqslant36k^2}f(x)= \min \{f(k), f(36k^2)\}.
\]

One can verify that when $n\geqslant72k^2+2k-1$,  $f(36k^2) \geqslant f(k)$, thus the lemma is true for this case.

\item{Case 2: There exist at least $n-1$ points of degree $\geqslant 6k$.}

From (\ref{eqn:lb-t23-incidences}) we know (for positive $k$),
\begin{multline*}
t_2+t_3 \geqslant 2+\frac{1}{6}\sum_{i\geqslant2} i\cdot r_i \geqslant 2+\left(\frac{1}{6}\right)6k(n-1)
\\ = kn-k+2 > k(n-k)-k(k-1).
\end{multline*}
Therefore, the lemma is true in both cases.

\end{proof}

We should mention that one could derive a similar bound for $t_2+t_3$ using a result of Kelly-Moser from \cite{KM58} 
(i.e., $t\geqslant kn-\frac{1}{2}(3k+2)(k-1)$ when at most $n-k$ are collinear and $n$ sufficiently large)
with Lemma \ref{lem:lb-23-pt-lines}.
Using the Kelly-Moser result one would need a larger value for $k$ but could achieve a better lower bound for $n$.
One could also derive a result similar to the corollary in the following section using this method.

\subsection{Circles Determined by Points in $\mathbb{R}^2$}

We shall now apply Theorem \ref{thm:2-3-pt-lines} to produce a corollary on the
number of circles determined by at most four points (among a finite set of points) in the plane.
The proof will utilize circular inversion.

Circular inversion is a transformation of the euclidean plane, i.e., a mapping of $\mathbb{R}^2\rightarrow\mathbb{R}^2$.
This transformation has several properties useful to the combinatorial geometer, e.g.,
for demonstrating the existence of circles determined by a finite set of points.
Motzkin may have been first to use circular inversion for this purpose.
More specifically, Motzkin (\cite{Mot51}) proved that in a finite set of points in the plane, 
not all collinear and not all cocircular, each point is incident to either a three-point line or a three-point circle.
Prompted by a conjecture of Erd\H{o}s, Elliott improved upon this by demonstrating a lower bound 
on the total number of circles determined by such a set of points \cite{Ell67}.
Using a similar method, B\'{a}lint and B\'{a}lintov\'{a} (\cite{BB94}) further extended the results of Elliott.
For any reader unfamiliar with circular inversion, we would recommend \cite[pp. 334--346]{Har00} for coverage of this transformation.

P. D. T. A. Elliott's 1967 result \cite{Ell67} that the number of circles is at least $\binom{n-1}{2}$ is slightly wrong. 
The lower bound is actually 
\[1+\binom{n-1}{2} - \left\lfloor\frac{n-1}{2}\right\rfloor \geqslant 1+\frac{1}{2}(n-1)(n-3),\] 
as may be seen by taking a circle with $n-1$ points 
on it and a point $p$ off the circle. 
It is easy to arrange the points such that $p$ lies on $\lfloor\frac{n-1}{2}\rfloor$ 
threepoint lines. (We discovered this when we tried to prove that the number of spheres is at least 
$\binom{n-1}{3}$ and discovered that there is a 
subtractive term derived from the Orchard Problem.) 
Apparently, this counter-example even escaped the notice of the well-known geometer Beniamino Segre
whom Elliott cited as providing an eight point counter-example to his result.
Elliott's proof can easily be modified to show the correct result with the same lower bound of 394 for n. 
By the way, B\'{a}lint and B\'{a}lintov\'{a} \cite{BB94}
gave the correct lower bound (i.e., $1+\frac{1}{2}(n-1)(n-3)$) in 1994, but without any explanation, and we (and also Elliott \cite{Ell08}) 
thought it was a misprint.

We define $Inv_p(S)$ to be the circular inversion of a point set $S$ about the point $p\notin S$. 
(In this case $S$ may be finite of infinite.)
Although we will not give a formal definition of circular inversion, 
we provide its basic properties (see \cite{Har00} for details):
\begin{itemize}
\item $Inv_p$ is self-inverse, i.e., $q = Inv_p(Inv_p(q))$.
\item If $l$ is a line passing through $p$, then $Inv_p(\l) = l$.
\item If $l$ is a line not passing through $p$, then $Inv_p(l)$
is a circle passing though $p$.
\end{itemize}

Let $S$ be a set of $n$ points, one of which is $p$.  Let $S' = Inv_p(S\backslash\{p\})$ be the set of $n-1$ points formed by circular inversion about $p$.
We begin with the following lemma. 

\begin{lem}
\label{lem:circ-3-4}
If at most $n-k$ points in $S$ are incident to any line or circle, then in $S'$, at most $(n-1)-k=n-k-1$ will be on any line.
Thus, (by Theorem \ref{thm:2-3-pt-lines}) when $n\geqslant 72k^2+2k$, the number of determined lines incident to 
at most three points in $S'$ is at least $k(n-k-1)-k(k-1)$.  
In $S'\cup\{p\}$, at most $(n-1)/2$ of these lines can be incident to $p$.
\end{lem}

\begin{cor}
Let $S$ be a set of $n$ points in $\mathbb{R}^2$ with at most $n-k$ ($k\geqslant 1$) points on any line or circle. 
If $n\geqslant 72k^2+2k$, then there exist at least $\frac{1}{8}(2k-1)(n^2-(2k+1)n)$ circles determined
in $S$.
\end{cor}

Let $c_r$ be the number of circles incident to $r\geqslant3$ points in $S$.
We denote by $c_r^{(p)}$ the number of circles in $S$ incident to $r$ points, one of which is $p$.
\begin{proof}
This corollary is true when there exists a circle or line incident to $n-1$ points (see the exceptional case for Elliot's result above), 
so assume $k\geqslant2$.
Let $S'=Inv_p(S\backslash\{p\})$, for some arbitrary $p\in S$.
By Lemma \ref{lem:circ-3-4}, at least $k(n-k-1)-k(k-1)$ determined lines in $S'$ are incident to at most three points.
Since at most $\frac{n-1}{2}$ of these lines are incident to $p$ in $S'\cup\{p\}$,
it must be the case that $c_3^{(p)}+c_4^{(p)}\geqslant k(n-k-1)-k(k-1)-(n-1)/2$. 
By repeating this argument for all $n$ points in $S$, each circle is counted at most four times.
Thus,
\begin{multline*}
c_3+c_4\geqslant\frac{n}{4}\left(k(n-k-1)-k(k-1)-\frac{n-1}{2}\right)\\
=\frac{1}{8}(2k-1)(n^2-(2k+1)n).
\end{multline*}

\end{proof}

\section{Planes Determined by Points in $\mathbb{R}^3$}
\label{sec-planes}
Let $S=\{p_1,p_2,\ldots,p_n\}$, a set of $n$ points in $\mathbb{R}^3$, no three collinear and not all coplanar. 

One method for determining the number of planes determined by a set of points in $\mathbb{R}^3$ 
is to take an arbitrary point, $p \in S$, 
and project from p onto a plane $\pi$, the other points of $S$.
The lines determined (by the projected points) on $\pi$ correlate to the planes determined by the point set.

\begin{defn}
Let $L$ be the set of $\binom{n}{2}$ lines determined by points of $S$.
Let $\pi$ be any plane which intersects every line in $L$, but does not contain any point of $S$.
Let $p_i\in S$ be the point from which we project, and let $l_{i,j}$ be the line connecting point $p_i$ with another point $p_j\in S$. 
The \emph{projection} of point $p_j$ on $\pi$ is the point of intersection of $l_{i,j}$ with $\pi$.
Since no three points are collinear, this projection of $p_j$ is unique with respect to $p_i$.
\end{defn}
Although no three points are collinear in $S$, there might exist three points whose projections onto $\pi$ are collinear.
Also, the point from which we project is not itself projected; So a set of $n$ points, will produce $n-1$ points projected onto $\pi$.

\begin{defn}
\label{def-lines-pi}
Let $t_k(p)$ be the number of lines determined on $\pi$, when projecting from point $p$, containing exactly $k$ $(\geqslant2)$ points.
Let $t(p)=\sum_{k\geqslant2}t_k(p)$.
\end{defn}
\begin{defn}
Let $m_k(p)$ be the number of planes containing exactly $k$ $(\geqslant3)$ points, one of which is point $p$.
Let $m(p)=\sum_{k\geqslant3}m_k(p)$.
\end{defn}

\begin{lem}
\label{lem-lines-to-planes}
For any point $p \in S$,
\begin{equation}
\label{eqn:lines-to-planes}
t_k(p) = m_{k+1}(p).
\end{equation}
\end{lem}

\begin{proof}
Let $p_n\in S$ be an arbitrary point from which we project the others onto a plane $\pi$.
Let $s_i$ and $s_j$ on $\pi$ be the projections of distinct points $p_i\in S$ and $p_j\in S$, respectively.
Let $l$ be the line determined by $s_i$ and $s_j$.
The points $p_i$, $p_j$ and $p_n$ determine a plane, $P$, which contains line $l$.
Any other point whose projection lies on $l$ is incident to $P$.
Since each point in $S-\{p_n\}$ has a unique projection onto $\pi$, the identity follows.
\end{proof}

This identity, and derivations thereof, will be utilized throughout this section to obtain our primary results.

\subsection{A Derivation from Melchior's Inequality}
Let $t_k$ be the number of lines determined by points in a plane containing exactly $k$ points.

We again use Melchior's Inequality \cite{Mel41},
\[
-3 \geqslant \sum_{k\geqslant2}(k-3)t_k.
\]

Applying this inequality to the projection onto $\pi$ from point $p_1 \in S$, along with the identity (\ref{eqn:lines-to-planes}), we see that
\[
-3 \geqslant \sum_{k\geqslant2}(k-3)t_k(p_1)=\sum_{k\geqslant3}(k-4)m_k(p_1).
\]

Since this inequality is true for whichever point we choose, we may extend this by summing over all $n$ points in $S$,
\[
-3n \geqslant \sum_{i=1}^n\sum_{k\geqslant3}(k-4)m_k(p_i)
\]

Let $m_k$ be the number of planes determined by points of $S$ containing exactly $k$ points.
Each $k$-plane is counted exactly $k$ times in this summation, yielding
\[
-3n \geqslant (-1)(3)m_3+(0)(4)m_4+(1)(5)m_5+(2)(6)m_6+\ldots.
\]

From this we get the following theorem:
\begin{thm}
\label{thm-melchior-sum}
Let $S$ be a set of $n$ points in $\mathbb{R}^3$, no three collinear and not all coplanar. 
Let $m_k$ be the number of planes determined by points of $S$ containing exactly $k$ points.
It must be the case that
\[
-3n \geqslant \sum_{k\geqslant3}k(k-4)m_k.
\]
\end{thm}

By separating the $m_3$ term from the summation and reordering the inequality, one can see our first corollary:
\begin{cor}
\label{cor:3-pt-planes}
Given a set of $n$ points in $\mathbb{R}^3$, no three collinear and not all coplanar, 
there exists at least $n$ planes containing exactly three points.
More specifically,
\[
m_3 \geqslant n + \sum_{k\geqslant4}\frac{k(k-4)}{3}m_k.
\]
\end{cor}

From Theorem \ref{thm-melchior-sum}, one may also derive a bound for the number of planes determined by at most four points.
The inequality from Theorem \ref{thm-melchior-sum} can be rewritten, 
\[
3n \leqslant 3m_3+0m_4-5m_5-12m_6-\ldots.
\]

Let $m$ be the total number of determined planes. By adding $5m$ to both sides, we see that
\[
5m+3n\leqslant 8m_3+5m_4+0m_5-7m_6-\ldots,
\]
or
\[
5m+3n\leqslant 8(m_3+m_4).
\]

From this we get our next corollary:
\begin{cor}
\label{cor:lb-34-pt-planes}
Let $S$ be a set of $n$ points in $\mathbb{R}^3$, no three collinear and not all coplanar. 
Let $m$ be the total number of planes determined by $S$.
The number of planes determined by $S$ containing at most four points is more than 
five-eighths of the total number of determined planes.
More specifically,
\[
	m_3+m_4\geqslant\frac{1}{8}(5m+3n).
\]
\end{cor}

\subsection{Planes Incident to Exactly Three Points}

In \cite{CS93}, Csima and Sawyer published their well-known result that among any set of $n\neq7$ points in a plane, not all collinear,
there exist $\geqslant\frac{6n}{13}$ lines incident to exactly two points.  We will now use this result to obtain
an analogous result for the number of planes determined by exactly three points in 3-space.

Let $S$ be a set of $n$ points in $\mathbb{R}^3$, no three collinear and not all coplanar. 
Let $p_1\in S$ be the point from which we project the others onto $\pi$.

Since the points are not all coplanar, their $n-1$ projections onto $\pi$ are not all collinear. Thus, Csima and Sawyer's 
result can be applied:
\[
m_3(p_1) = t_2(p_1) \geqslant \frac{6(n-1)}{13}
\]

By summing the above inequality for all $n$ points, we would count each three-point plane three times.  
By taking one-third of that total, we arrive at the following:
\begin{multline*}
m_3 = \frac{1}{3}(m_3(p_1)+m_3(p_2)+\ldots+m_3(p_n)) \geqslant \frac{2n(n-1)}{13} = \frac{4}{13}\binom{n}{2}
\end{multline*}

\begin{thm}
\label{thm-3-pt-planes}
Let $S$ be a set of $n$ points in $\mathbb{R}^3$, no three collinear and not all coplanar.
There exists at least $\frac{4}{13}\binom{n}{2}$ planes determined by exactly three points.
\end{thm}

\subsection{Total Number of Determined Planes}

The following lemma is an extension of a result of Erd\H{o}s and Purdy (called ``Lemma 2'') in \cite{EP77}.

\begin{lem}
\label{lem-lb-planes}
Let $S$ be a set of $n$ points in $\mathbb{R}^3$, not all coplanar and no three collinear.
Let $m$ be the total number of planes determined by $S$.
If exactly $n-k$ of the points are coplanar, then
\[
m \geqslant 1 + k \binom{n-k}{2}-\binom{k}{2}\left(\frac{n-k}{2}\right).
\]
\end{lem}

\begin{proof}
\item{Case $k=1$:}
Let $S$ be the set of $n$ points. Let $M$ be the plane containing $n-1$ points, and $p$ the point not on $M$.  
Each pair of points on $M$, along with $p$, determine a three-point plane.  The relation is true with equality.
\item{Case $k=2$:}
Let $p$ and $q$ be the two points not on $M$. The line $\overline{pq}$ intersects $M$ at a point $r\notin S$.
There can be at most $\lfloor\frac{n-2}{2}\rfloor$ pairs that determine a line through $r$, and therefore at most $\lfloor\frac{n-2}{2}\rfloor$
pairs that determine a plane with $p$ that is incident to $q$ (or vice versa).  Therefore, the number of planes 
determined is at least $1 + 2 \binom{n-2}{2}-\left\lfloor\frac{n-2}{2}\right\rfloor$.
\item{Case $k\geqslant3$:}
Following the same argument as the case for $k=2$, since there are $\binom{k}{2}$ pairs of points not on $M$, the number of planes 
determined is at least $1 + k \binom{n-k}{2}-\binom{k}{2}\left\lfloor\frac{n-k}{2}\right\rfloor$.
\end{proof}

Before we begin the primary result of this section, we will also need the following lemma.
\begin{lem}
\label{lem-total-planes}
Let $m_k$ be the number of planes incident to exactly $k$ points.  
Let $m$ be the total number planes determined by a point set.  
Given a set of $n$ points, no three collinear and not all coplanar, 
\[
6m \geqslant 3n+\sum_{k\geqslant3}\binom{k}{2}m_k
\]
\end{lem}
\begin{proof}
Theorem \ref{thm-melchior-sum} states the following,
\begin{equation}
\label{eqn-mel-pairs}
-3n \geqslant \sum_{k\geqslant3}(k^2-4k)m_k = \sum_{k\geqslant3}(k^2-k)m_k - 3\cdot\sum_{k\geqslant3}k\cdot m_k
\end{equation}
By negating this inequality one gets, 
\begin{equation}
\label{eqn-neg-mel-pairs}
3m_3+0m_4-5m_5-12m_6-21m_7-32m_8\ldots\geqslant3n
\end{equation}
Similarly,  one can unwind the summation from (\ref{eqn-mel-pairs}) to get,
\begin{multline}
\label{eqn-exp-mel-pairs}
9m_3 + 12m_4+15m_5+18m_6+21m_7+24m_8+\ldots\\ \geqslant 3n+\sum_{k\geqslant3}(k^2-k)m_k
\end{multline}
Adding (\ref{eqn-neg-mel-pairs}) to (\ref{eqn-exp-mel-pairs}) produces,
\[
12m_3+12m_4+10m_5+6m_6\geqslant6n+\sum_{k\geqslant3}(k^2-k)m_k
\]
Therefore,
\[
12m\geqslant12(m_3+m_4+m_5+m_6)\geqslant6n+\sum_{k\geqslant3}(k^2-k)m_k
\]
Dividing the inequality by two produces the lemma.
\end{proof}

This leads us to the following theorem, which is a generalization of the Kelly-Moser Theorem (called ``Theorem 4.1'' in \cite{KM58}) to three dimensions:
\begin{thm}
\label{thm-total-planes}
Let $S$ be a set of $n$ points in $\mathbb{R}^3$, no three collinear and at most $n-k$ coplanar.
If $n\geqslant g(k)\buildrel\rm def\over= 54k^2+\frac{9}{2}k$, the total number of planes determined 
by $S$ is at least $1 + k \binom{n-k}{2}-\binom{k}{2}\left(\frac{n-k}{2}\right)$.
\end{thm}

The function, $f(k)\buildrel\rm def\over= 1 + k \binom{n-k}{2}-\binom{k}{2}\left(\frac{n-k}{2}\right)$, 
is a cubic polynomial of $k$ and can be rewritten as:
\[
f(k)=\frac{3}{4}k^3+\frac{1}{4}(1-5 n) k^2 +\frac{1}{4}n (-1+2 n) k + 1
\]

Let $c_1$ and $c_2$, where $c_1<c_2$, be the function's two local extrema at 
$\frac{1}{9} \left(-1+5 n\pm\sqrt{1-n+7 n^2}\right)$.
For all $n\geqslant4$, $f(c_1)>0$ and $f(c_2)<0$.  Furthermore, $f''(c_1)<0$ and $f''(c_2)>0$.

\begin{proof}[Proof of Theorem \ref{thm-total-planes}:]
Trivially, $S$ contains $\binom{n}{2}$ pairs of points.
We define the degree for a pair of points to be the number of determined planes to which the pair is incident.
\item{Case 1: More than $\frac{n}{2}$ pairs of points have degree $<6k$.}

These ($>\frac{n}{2}$) pairs cannot form a matching, hence two pairs of low degree ($<6k$) must share a point.
Assume two such pairs are $\{p,q\}$ and $\{p,r\}$. Let $M$ be the plane determined by the points $p$, $q$, and $r$.
Let $a<6k$ be the number of planes determined by $S$, not including $M$, incident to the pair $\{p,q\}$.
Likewise, let $b<6k$ be the number of planes incident to the pair $\{p,r\}$.
Since no three points are collinear, any plane passing through $\{p,q\}$ can share at most one point of $S$, 
other than $p$, with a plane though the pair $\{p,r\}$.
Therefore, at most $a\cdot b < 36k^2$ points of $S$ are not on $M$.
If $M$ has exactly $n-x$ points on it, then $k \leqslant x < 36k^2$. 

We claim that for all $n \geqslant g(k)$, the following two conditions are true:
\begin{list}{$\circ$}{}
\item $36k^2 < c_2$, where $c_2$ is the second local extremum of $f(k)$.
\item $f(36k^2) \geqslant f(k)$
\end{list}
To verify the first, it would be sufficient for $(5+\sqrt{6})n\geqslant324k^2+1$, and thus, it is true for all $n \geqslant 44k^2$.
To verify the second, one must first consider our formula as a function of two variables, $n$ and $k$, i.e.,
\[f_1(n,k) \buildrel\rm def\over= 1 + k \binom{n-k}{2}-\binom{k}{2}\left(\frac{n-k}{2}\right).\]
Considering $k$ to be a constant,  
\[f_2(n) \buildrel\rm def\over= f_1(n,36k^2)-f_1(n,k)\]
is a convex quadratic function of $n$.  
By solving $f_2(n)=0$, one can see that for all $n \geqslant \lceil54k^2+\frac{9}{2}k\rceil$, 
$f_2(n)$ is positive.

From the two properties listed above, it is obvious that for all $x$ such that $k \leqslant x < 36k^2$, 
it is also true that $f(x) \geqslant f(k)$. So our inequality holds in this first case.

\item{Case 2: At most $\frac{n}{2}$ pairs of points have degree $<6k$.}

There are at least $\binom{n}{2}-\frac{n}{2}=\frac{1}{2}n(n-2)$ pairs of degree $\geqslant6k$.  
Let $P_i$ be the number of pairs of points incident to exactly $i$ planes.
By Lemma \ref{lem-total-planes},
\[
6m > \sum_{k\geqslant3}\binom{k}{2}m_k 
= \sum_{i\geqslant2}i\cdot P_i \geqslant \frac{1}{2}n(n-2)(6k)
\]
Thus, for all $k\geqslant1$ (and $n\geqslant4$), 
\[
m > \frac{1}{2}n(n-2)k \geqslant \frac{1}{2}(n-k)(n-k-1)k + 1 = 1+k\binom{n-k}{2}.
\]
Thus our inequality holds in this case as well.
\end{proof}

\subsection{Planes Determined by At Most Four Points}

\begin{lem}
Let $m_3$ be the number of planes incident to exactly three points.
Given a set of points $S$, no three collinear, if $r$ points lie on a plane $\pi$, 
and $s$ do not, 
then $m_3 \geqslant s\binom{r}{2}-\frac{1}{2}rs(s-1)$.
\end{lem}
\begin{proof}
We define a plane with more than three points to be \emph{spoiled}.
This lemma is obviously true for $s=0$ and $s=1$. 
We will use induction on $s$. 

Assume $s > 1$.  Let $p$ be one of the $s$ points not on $\pi$. 
If $p$ is removed
there are $(s-1)\binom{r}{2}-\left(\frac{r}{2}\right)(s-1)(s-2)$
three-point planes.  The addition of $p$ will spoil at most 
$\left(\frac{r}{2}\right)(s-1)$ of those planes.  

Let $q$ be any one of the other $s-1$ points, in $S$ and not on $\pi$. 
Let $x\notin S$ be the point (possibly at $\infty$) at which the line 
$pq$ intersects $\pi$. 
There can be \emph{at most} $\frac{r}{2}$ pairs of points, in $S$ and on $\pi$,
that determine a line through $x$, thus forming a four point plane with $p$ and $q$.

So, the addition of $p$ introduces $\binom{r}{2}$ new planes of which
at most $\left(\frac{r}{2}\right)(s-1)$ contain four or more points.
Thus, 
\begin{multline*}
m_3\geqslant\\
(s-1)\binom{r}{2}-\left(\frac{r}{2}\right)(s-1)(s-2)-\left(\frac{r}{2}\right)(s-1)+\binom{r}{2}-\left(\frac{r}{2}\right)(s-1)\\
= s\binom{r}{2}-\left(\frac{1}{2}\right)rs(s-1)
\end{multline*}
\end{proof}

\begin{thm}
\label{thm-34-planes}
Let $S$ be a set of $n$ points in $\mathbb{R}^3$, no three collinear and at most
$n-k$ coplanar.  If $n\geqslant g(k)\buildrel\rm def\over= \left(184+\frac{8}{25}\right)k^2+4k$, then 
\[
m_3+m_4\geqslant k\binom{n-k}{2}-(n-k)\binom{k}{2}
\]
\end{thm}

We now define $f(k) \buildrel\rm def\over= k\binom{n-k}{2}-(n-k)\binom{k}{2}$.
The function can be rewritten in the following form:
\[
	f(k) = k^3 - \frac{3n}{2} k^2 + \frac{n^2}{2} k
\]
Let $c_1$ and $c_2$, where $c_1<c_2$, be the two local extrema of this function at $\frac{1}{6} (3\pm\sqrt{3}) n$.
We note that $f(c_1)>0$  (i.e. $f(c_1)= \frac{\sqrt{3}}{36}n^3$) and $f(c_2)<0$.
Furthermore, $f''(c_1)<0$ and $f''(c_2)>0$.

For obvious reasons, the following proof is very similar to the 
proof for Theorem \ref{thm-total-planes} in the previous section.

\begin{proof}[Proof of Theorem \ref{thm-34-planes}:]
Trivially, $S$ contains $\binom{n}{2}$ pairs of points.
We define the degree for a pair of points to be the number of determined planes to which the pair is incident.

\item{Case 1: More than $\frac{n}{2}$ pairs of points have degree $<\frac{48}{5}k$.}

These ($>\frac{n}{2}$) pairs cannot form a matching, hence two pairs must share a point.
Assume two such pairs are $\{p,q\}$ and $\{p,r\}$. Let $M$ be the plane determined by the points $p$, $q$, and $r$.
As demonstrated in Theorem \ref{thm-total-planes}, there can be at most 
$(\frac{48}{5}k)^2=(92+\frac{4}{25})k^2$ points of $S$ not on $M$.
If $M$ has exactly $n-x$ points on it, then $(\frac{48}{5}k)^2 \geqslant x \geqslant k$. 

We claim that for all $n \geqslant g(k)$, the following two conditions are true:
\begin{list}{$\circ$}{}
\item $(\frac{48}{5}k)^2 < c_2$, where $c_2$ is the second local extremum of $f(k)$.
\item $f\left((\frac{48}{5}k)^2\right) \geqslant f(k)$
\end{list}
To verify the first, it would be sufficient for $(3+\sqrt{3})n\geqslant558k^2$, and thus, it is true for all $n \geqslant 118k^2$.
To verify the second, one must first consider our formula as a function of two variables, $n$ and $k$, i.e.
\[f_1(n,k) \buildrel\rm def\over= k\binom{n-k}{2}-(n-k)\binom{k}{2}.\]
Considering $k$ to be a constant,
\[f_2(n) \buildrel\rm def\over= f_1\left(n,\left(\frac{48}{5}k\right)^2\right)-f_1(n,k)\]
is a convex quadratic function of n.  
By solving $f_2(n)=0$, one can see that for all $n \geqslant \lceil\left(184+\frac{8}{25}\right)k^2+4k\rceil$ (and $k\geqslant1$), 
$f_2(n)$ is positive.

From the two properties listed above, it is obvious that for all x such that $k \leqslant x < (\frac{48}{5}k)^2$, 
it is also true that $f(x) \geqslant f(k)$. So our inequality holds in this first case.

\item{Case 2: At most $\frac{n}{2}$ pairs of points have degree $<\frac{48}{5}k$.}

There are at least $\binom{n}{2}-\frac{n}{2}=\frac{1}{2}n(n-2)$ pairs of degree $\geqslant\frac{48}{5}k$.
Let $P_i$ be the number of pairs incident to exactly $i$ planes.  By Corollary \ref{cor:lb-34-pt-planes}, 
we know that $\frac{48}{5}(m_3+m_4)=\frac{8}{5}\cdot6\cdot(m_3+m_4)\geqslant 6m$. 
Thus, by Lemma \ref{lem-total-planes}
\[
\left(\frac{48}{5}\right)(m_3+m_4) > \sum_{k\geqslant3}\binom{k}{2}m_k 
= \sum_{i\geqslant2}i\cdot P_i \geqslant \frac{1}{2}n(n-2)\left(\frac{48}{5}\right)k
\]
Therefore, $m_3+m_4 > \frac{1}{2}n(n-2)k\geqslant k\binom{n-k}{2}$. So the theorem is true in this case as well.
\end{proof}

\subsection{Corollaries and Conjectures}
\label{sec:cor-conj}

From Theorem \ref{thm-total-planes}, we have two corollaries.  Both relate to conjectures found in \cite[p. 815]{EP95}.
In \cite{EP77}, Erd\H{o}s and Purdy proved $m \geqslant 1+\binom{n-1}{2}$ for $n\geqslant552$.
Theorem \ref{thm-total-planes} improves this by showing it to be true for $n\geqslant 59$ (i.e., $k=1$), 
which leads to the following corollary.
\begin{cor}
Let $S$ be a set of $n\geqslant 59$ points in $\mathbb{R}^3$, no three collinear and not all coplanar.
Let $m$ be the number of planes determined by $S$, and $t$ the number of lines. For all such point sets,
$m - t + n \geqslant 2$.
\end{cor}
\begin{proof}
Putting $k=1$, Theorem \ref{thm-total-planes} says that $m \geqslant 1+\binom{n-1}{2}$.
Since no three points are collinear, $t=\binom{n}{2}=\binom{n-1}{2}+(n-1)$. Thus, $m-t+n\geqslant2$.
\end{proof}

This is conjectured by Purdy to be true when $n\geqslant32$ for any finite set of points in $\mathbb{R}^3$ that are not all coplanar 
and not all on two skew lines.

\begin{cor}
Let $S$ be a set of $n\geqslant 225$ points in $\mathbb{R}^3$, no three collinear and no $n-1$ coplanar.
Let $m$ be the number of planes determined by $S$, and $t$ the number of lines. For all such point sets, $m\geqslant t$.
\end{cor}
\begin{proof}
Putting $k=2$, thus $n\geqslant225$, Theorem \ref{thm-total-planes} shows that 
\[
m\geqslant 2 \binom{n-2}{2}-\frac{n-2}{2}\geqslant\binom{n}{2}=t
\] 
(for $n\geqslant9$).
\end{proof}

Purdy proved in \cite{Pur86} that if the points are not all coplanar and not all on two skew lines 
then $m\geqslant ct$, for some $c>0$.
Erd\H{o}s asked what are sufficient conditions for $m\geqslant t$.  
Purdy conjectured in \cite{EP95}
that $m\geqslant t$ when $n$ is sufficiently large, no $n-1$ points are coplanar and the points do not lie on two skew lines.

This conjecture is easily seen to be false for projective geometries over finite fields.
Let $q=p^k$, where $p$ is any prime and $k\geqslant1$.
We denote by $PG(d,q)$ the projective $d$-space over $GF(q)$,
and by ${m \brack k}_q$, we denote the Gaussian coefficient.
In $PG(3, q)$, the number of points (same as the number of planes) is ${4 \brack 3}_q = q^3+q^2+q+1$,
and the number of lines is ${4 \brack 2}_q=(q^2+1)(q^2+q+1)=q^4+q^3+2q^2+q+1$ \cite[p. 66]{Hir79} \cite[p. 168]{Oxl92}.
(From this one can see that $n < t > m$.)
Since $PG(3, q)$ is a rank-4 matroid, the conjecture is also false for matroids.

This conjecture has also been extended by Purdy to $d$-dimensional space \cite[p. 815]{EP95}.  We define the rank of a flat (i.e., $rk(f)$) 
to be one more than its dimension, e.g., points have rank one, lines rank two, etc.
A set of flats $F=\{f_1, f_2, \ldots, f_r\}$ is defined to be a \emph{covering} set of flats for a point set $S\subset \mathbb{R}^d$ if every point in $S$ is incident to a flat in $F$.
Furthermore, the rank of a set of flats $F$ is defined to be the sum of the ranks of its members, i.e., $rk(F)=rk(f_1)+rk(f_2)+\ldots+rk(f_r)$.
A point configuration $S$ is \emph{irreducible} if there does not exist a \emph{covering} set of $r$ ($\geqslant 2$) flats 
with rank at most $d+1=rk(\mathbb{R}^d)$.

Let $w_k$ be the number of determined flats of rank $k$.
For any sufficiently large irreducible point configuration in $\mathbb{R}^d$, Purdy conjectures that $w_d\geqslant w_{d-1}$.
By projection (i.e., one which preserves the number of $(d-2)$-flats), it would follow that
$w_d\geqslant w_{d-1} \geqslant w_{d-2} \geqslant \ldots \geqslant w_1$,
thus implying unimodality in this case. 

Seymour proved \cite{Sey82}, in the more general context of matroids, that if no five points are collinear then $t^2\geqslant mn$.
This is related to a conjecture by Mason \cite{Mas72} and others that the sequence of Whitney Numbers
(i.e., $w_1, w_2, w_3,\ldots$) is log-concave,
i.e., $w_i^2 \geqslant w_{i-1}w_{i+1}$ for all $i>0$.
Purdy proved in $\mathbb{R}^3$ \cite{Pur86} that if the points do not all lie on a plane then $t^2 \geqslant c m n$, for some $c>0$.
See \cite{EP95} for further discussion of these and other conjectures.

\section{Spheres Determined by Points in $\mathbb{R}^3$}

Circular Inversion has a lesser known extension to higher dimensional space, i.e., spherical inversion.
For the convenience of the reader, we provide below a basic definition and the relevant properties of spherical inversion.  
We refer the reader to \cite[pp. 83--87]{Bla00} for a rigorous definition and proof of the given properties.

Let $p$ be any fixed point in $\mathbb{R}^3$. Without loss of generality, let $p$ be the origin of a coordinate system. 
We assume, of course, that coordinates are determined relative to some orthonormal basis.
Let $q$, with coordinates $(x, y, z)$, be any other point in $\mathbb{R}^3$.
We define the ``spherical inverse'' to be (for any $q \neq p$): 
\[
	Inv_p(q)\buildrel\rm def\over= \frac{q}{||q||^2} = \left(\frac{1}{x^2+y^2+z^2}\right) (x, y, z)
\]
(This inversion occurs about a sphere of unit radius centered at $p$. $Inv_p(p)$ is left undefined.)
We apply this mapping, i.e., $Inv$, not only to points but also to sets of points, either infinite or finite, intending the obvious results.

The mapping has the following properties (for any arbitrary point $p$):
\begin{itemize}
\item $Inv_p$ is self-inverse, i.e., $q = Inv_p(Inv_p(q))$.
\item If $\pi$ is a plane passing through $p$, then $Inv_p(\pi) = \pi$.
\item If $\pi$ is a plane not passing through $p$, then $Inv_p(\pi)$
is a sphere passing though $p$.
\end{itemize}

\begin{lem}
\label{lem-inv-props}
Let $S$ be a set of $n\geqslant5$ points in $\mathbb{R}^3$, not all cospherical, no $n-1$ coplanar, no four cocircular and no three collinear.
For any arbitrary $p\in S$, the set $Inv_p(S \backslash \{p\})\cup\{p\}$ will be likewise.
\end{lem}
\begin{proof}
For contradiction, assume $q_1$, $q_2$ and $q_3$ are three collinear points in $S$.
Choose the center of inversion, $p$, to be any other point in $S$.
Let $l$ be the line determined by the $q_i$.
Let $M$ be the plane incident to $l$ and passing through $p$ (the center of inversion).
Let $N$ be any plane incident to $l$ but not passing through $p$.
Obviously $Inv_p(l) = Inv_p(M) \cap Inv_p(N)$, which is the intersection of a 
plane (i.e., $Inv_p(M)$) with a sphere (i.e., $Inv_p(N)$).
Thus, $Inv_p(l)$ is a circle incident to these four points.
By self-inversion, one can see that the converse is also true.

The other claims are obvious.
\end{proof}

\subsection{Spheres Incident To Exactly Four Points}

Using the lemma above and some careful counting, we now prove a lower bound on the number of four point spheres 
determined by a set of points in $\mathbb{R}^3$.  We also utilize a result, from an earlier section, 
concerning the number of three-point planes determined by a set of points.
\begin{thm}
Let $S$ be a set of $n\geqslant5$ points in $\mathbb{R}^3$, not all cospherical or coplanar, no four cocircular and no three collinear.
There exist at least $\epsilon\binom{n}{3}$ spheres incident to exactly four points of $S$, where $\epsilon=\frac{9}{208}$.
\end{thm}
\begin{proof}
Choose any $p\in S$ to be the center of inversion.
Let $S'=Inv_p(S \backslash \{p\})$, and let $S''=S'\cup\{p\}$.

Assume $n-1$ points in $S$ are coplanar.  Let $p$ be the point not on the plane.  
Since no three points are collinear and no four cocircular, there are $\binom{n-1}{3}$ 
four point spheres determined through $p$, and the theorem is true.
We now assume that no $n-1$ points are coplanar in $S$.

By Lemma \ref{lem-inv-props}, $S''$ is a set of $n$ points such that no $n-1$ are coplanar.
Thus, the points of $S'$ are not all coplanar (and no three collinear).
From Theorem \ref{thm-3-pt-planes}, we know that $S'$ determines at least $\frac{2}{13}(n-1)(n-2)$ three-point planes.
There are two possible cases.
\begin{itemize}
\item{Case 1: At most $\frac{1}{8}(n-1)(n-2)$ of these three-point planes (determined by points in $S'$) pass through $p$.}

The set $S$ thus contains at least $(\frac{2}{13}-\frac{1}{8})(n-1)(n-2)=\frac{3}{104}(n-1)(n-2)$ four point spheres incident to point $p$.
\item{Case 2: More than $\frac{1}{8}(n-1)(n-2)$ of these three-point planes (determined by points in $S'$) pass through $p$.}

We now carefully count the $\binom{n-1}{2}=\frac{1}{2}(n-1)(n-2)$ pairs of points in $S'$ (as seen from $p$).
Each of the at least $\frac{1}{8}(n-1)(n-2)$ planes that pass through $p$ account for at least three of these pairs, 
i.e., at least $\frac{3}{8}(n-1)(n-2)$ pairs.
This leaves at most $(\frac{1}{2}-\frac{3}{8})(n-1)(n-2)=(\frac{1}{8})(n-1)(n-2)$ pairs to form a three-point plane with $p$ in $S''$.
Hence, the set $S''$ must determine at least $(\frac{2}{13}-\frac{1}{8})(n-1)(n-2)=\frac{3}{104}(n-1)(n-2)$ 
three-point planes that do not pass through $p$.

Thus, the set $S$ (again) contains at least $\frac{3}{104}(n-1)(n-2)$ four point spheres incident to point $p$.
\end{itemize}

Since the selection of $p$ was arbitrary this argument may be repeated $n$ times, counting each of these spheres four times.
Thus the number of four point spheres determined by $S$ is at least 
$(\frac{1}{4})(\frac{3}{104})n(n-1)(n-2) = \frac{9}{208}\binom{n}{3}$.
\end{proof}

\subsection{The Orchard Problem}
\label{sec-orchard}

Let $t_3(S)$ be the number of three-point lines determined by some finite point set $S$ in the plane.
Let $t_3^{orchard}(n)$ be the maximum of $t_3(S)$ for all $n$ element point sets $S$ in the plane containing no four collinear points.
We refer to determining the exact value of $t_3^{orchard}(n)$ as the ``classic'' Orchard Problem.

Consider a set of $n$ points in three-space, no three collinear and no four cocircular.
Our ultimate aim is to prove that the minimum number of spheres determined by 
these points is exactly $1+\binom{n-1}{3}-t_3^{orchard}(n-1)$. 

To this end, we let $S$ be a set of $n$ cospherical points with no four cocircular.
We shall prove in this section that the maximum number of planes, determined by $S$, that share a common point $p\notin S$ 
is equal to $t_3^{orchard}(n)$. For brevity, we will refer to this bound as $\mathcal{M}_3^{max}(n)$.

We begin with the lemma below that will be used in the following section (i.e., Section \ref{sec:sphere-at-most-5})
to prove a lower bound on the total number of determined spheres.
It is this lemma that led us to the rather unexpected equivalence to the Orchard Problem.

\begin{lem}
\label{lem-cospherical-1-off}
Let $S$ be a set of $n$ points in $\mathbb{R}^3$, exactly $n-k$ cospherical, no four cocircular and no three collinear.
Let $\sigma_0$ be the sphere to which $n-k$ points are incident.
Let $p\in S$ be any point not on $\sigma_0$.
There are at least $\binom{n-k}{3}-\mathcal{M}_3^{max}(n-k)$ spheres 
determined by $p$ and three points on $\sigma_0$.  
Furthermore, each of these spheres is incident to at most $3+k$ points of $S$.
\end{lem}
\begin{proof}
From $p$, project the $n-k$ points of $\sigma_0$ onto a plane $\pi$.
Let $t_k(p)$ be defined as in Definition \ref{def-lines-pi} for the $n-k$ points being projected.
Since no four points are cocircular, $t_k(p)=0$ for all $k\geqslant4$. 
So obviously, $t_3(p) \leqslant \mathcal{M}_3^{max}(n-k) \leqslant t_3^{orchard}(n-k) \leqslant \frac{1}{3}\binom{n-k}{2}$. 
All point triples on $\sigma_0$ not forming a plane through $p$ must determine a sphere with $p$.
\end{proof}

In the proof above, we assert that obviously $M_3^{max}(n)\leqslant t_3^{orchard}(n)$.
Combinatorially, one can see that $t_3^{orchard}(n)\leqslant\frac{1}{3}\binom{n}{2}$, 
since no two distinct three-point lines may share a pair of points.
This is further improved by observing that, by counting all pairs of points, $3t_3+t_2=\binom{n}{2}$.
Using the Csima-Sawyer result (i.e., $t_2\geqslant \frac{6n}{13}$),
this upper bound can be lowered to $\lfloor\frac{1}{6}n^2-\frac{25}{78}n\rfloor$.

Improving upon Sylvester's work, point configurations have been demonstrated by Burr, Gr\"{u}nbaum, and Sloane \cite{BGS74}, 
and again by F\"{u}redi and Pal\'{a}sti's \cite{FP84} that come very close to this theoretical bound. 
More specifically, these configurations have shown that $t_3^{orchard}(n)\geqslant\lfloor\frac{1}{6}n^2-\frac{1}{2}n+1\rfloor$.
See \cite[pp. 315--318]{BMP05} for coverage of more recent results.

A set of $n$ points on a sphere, no four cocircular will project onto $n$ points on a plane, 
no four collinear, so there will be at most $t_3^{orchard}(n)$ three-point lines and consequently at most 
that many planes through three points on the sphere and the point of projection. 
However, to show these two bounds to be equivalent we need to also prove the more difficult ``converse'' of this.
That is, we must prove that every orchard configuration in the plane can back project into points on a sphere, no four cocircular. 
From this equivalence, we conclude that not only is $1+\binom{n-1}{3}-t_3^{orchard}(n-1)$ the lower bound for determined spheres 
(see Section \ref{sec:sphere-at-most-5}), but that this bound is always attainable.
(Note that the projection of a circle onto a plane is a conic section, 
and that five points, no three collinear, are required to determine a unique conic section \cite[p. 85]{Cox74}.)

\begin{thm}
Let $S$ be a configuration of $n$ points, no four collinear, on a plane $\pi$.
There must exist a set $S'$ of $n$ points, all cospherical but no four cocircular, 
and a point of projection $p$ such that projecting the points from $S'$ onto $\pi$
produces the point set $S$.
\end{thm}

\begin{proof}
Let $\pi$ be the plane in $\mathbb{R}^3$ containing the point configuration $S$. 
We will assign each point in $S$ unique coordinates $(x, y, 1)$, for some $x$,$y \in \mathbb{R}$.
We further require the points of $S$ to be rotated such that no two points share a first coordinate.

Let $\overline{w}=(w_1,w_2,1)$, $\overline{x}=(x_1,x_2,1)$, $\overline{y}=(y_1,y_2,1)$ and 
$\overline{z}=(z_1,z_2,1)$ be any four points of the configuration $S$.
Let $\sigma_{\overline{p}}$ be a unit sphere, centered at $\overline{p}=(p_1,p_2,0)$ and tangent to $\pi$, to which the points of $S'$ are incident.
The point from which we project will be the center of the sphere, i.e., $\overline{p}$.

If any three of the four points are collinear they will not determine a conic section, so we assume that this is not the case.
We also assume that $\overline{x}$, $\overline{y}$ and $\overline{z}$ are labelled in a clockwise order, as seen from $p$.

Let $a=||\overline{w}-\overline{p}||$, 
$b=||\overline{x}-\overline{p}||$, $c=||\overline{y}-\overline{p}||$, 
and $d=||\overline{z}-\overline{p}||$
be the euclidean distances from each of the four points on $\pi$
to the center of the sphere $\sigma_{\overline{p}}$.
Let $\hat{w}=\frac{1}{a}(\overline{w}-\overline{p})$, $\hat{x}=\frac{1}{b}(\overline{x}-\overline{p})$, 
$\hat{y}=\frac{1}{c}(\overline{y}-\overline{p})$, and $\hat{z}=\frac{1}{d}(\overline{z}-\overline{p})$ 
be unit vectors representing four points on the sphere which project to $\overline{w}$, $\overline{x}$, 
$\overline{y}$ and $\overline{z}$, respectively.

Let $m_{\overline{p}}(\overline{x}, \overline{y}, \overline{z})$ be the plane determined by points 
at $(\hat{x}+\overline{p})$, $(\hat{y}+\overline{p})$ and $(\hat{z}+\overline{p})$.
The plane $m_{\overline{p}}(\overline{x}, \overline{y}, \overline{z})$ (which does not contain $\overline{p}$) 
intersects $\sigma_{\overline{p}}$ in a circle.
Let $c_{\overline{p}}(\overline{x}, \overline{y}, \overline{z})$ be the circle formed by the intersection of 
$m_{\overline{p}}(\overline{x}, \overline{y}, \overline{z})$ with $\sigma_{\overline{p}}$.

We denote by $(\hat{u},\hat{v})$ the scalar product of two vectors $\hat{u}$ and $\hat{v}$.
We now consider $\overline{w}$, $\overline{x}$, $\overline{y}$ and $\overline{z}$ to be fixed.
Let $\overline{c}=(\hat{y}-\hat{x})\times(\hat{z}-\hat{y})$. 
Obviously, $\overline{c}$ is a vector orthogonal to the plane $m_{\overline{p}}(\overline{x}, \overline{y}, \overline{z})$.
For any two points, e.g.,  $\hat{u}$ and $\hat{v}$, on the circle $c_{\overline{p}}(\overline{x}, \overline{y}, \overline{z})$,
their scalar products (as vectors relative to the center) with $\overline{c}$ are the same, i.e., $(\hat{u}-\overline{p}, \overline{c})=(\hat{v}-\overline{p}, \overline{c})$.
Therefore, $\hat{w}$ is on $c_{\overline{p}}(\overline{x}, \overline{y}, \overline{z})$ if and only if 
$f(p_1,p_2)\buildrel\rm def\over=(\hat{w}, \overline{c})-(\hat{x}, \overline{c}) = 0$.

The function $f(p_1,p_2)$ can be rewritten as
$(Aa+Bb+Cc+Dd)/(a\cdot b\cdot c\cdot d)$, where $A$,$B$,$C$ and $D$ are constants 
(i.e., their values are constant relative to the points $\overline{w}$, $\overline{x}$, $\overline{y}$ and $\overline{z}$).
Since we are concerned with the roots of the function $f$, we will only consider its numerator.
We also fix $p_2$ to be zero.

Let $g_0(p_1)\buildrel\rm def\over=(a\cdot b\cdot c\cdot d)\cdot f(p_1,0)=Aa+Bb+Cc+Dd$, and thus, 
\begin{multline*}
g'_0(p_1)=\frac{A(w_1-p_1)}{\sqrt{1+(w_1-p_1)^2+w_2^2}}+\frac{B(x_1-p_1)}{\sqrt{1+(x_1-p_1)^2+x_2^2}}\\
+\frac{C(y_1-p_1)}{\sqrt{1+(y_1-p_1)^2+y_2^2}}+\frac{D(z_1-p_1)}{\sqrt{1+(z_1-p_1)^2+z_2^2}}.
\end{multline*}
The first term of $g'_0(p_1)$ has two conjugate poles which are the complex zeroes of the monic quadratic 
$1+(w_1-p_1)^2+w_2^2=p_1^2 -2w_1p_1 + 1 + w_1^2 + w_2^2$.
The quadratic polynomials are all distinct, and since their roots occur in conjugate pairs, their roots are all distinct.
So, $g'_0(p_1)$ has eight poles, implying that $g_0(p1)$ is not identically zero.

Let $g_1, g_2,\ldots,g_7$ be conjugates of $g_0$ 
formed by reversing (in the seven possible ways remaining) the signs of its last three terms.
By taking the product of these eight functions we obtain a polynomial (i.e., with no radicals).
Let $P(p_1)=g_0(p_1)\cdot g_1(p_1) \cdot \ldots \cdot g_7(p_1)$.
Obviously, $P(p_1)$ is zero wherever $g_0(p1)$ is zero and furthermore, $P(p_1)$ is not identically zero.
Thus, $P(p_1)$ is a polynomial having at most a finite number of (real) roots.

Since there are only a finite number of four point combinations in $S$, there must exists a 
point $(p_1, 0, 1)$ on $\pi$ to which the sphere may be tangent such that the points of $S'$, 
no four cocircular, project onto the points of $S$.  Therefore, $t_3^{orchard}(n)=\mathcal{M}_3^{max}(n)$.
\end{proof}

\subsection{Total Number of Spheres Determined}
\label{sec:sphere-at-most-5}

In this section, we shall prove the following.
\begin{thm}
\label{thm-sph-5-pt}
Let $S$ be a set of $n$ points in $\mathbb{R}^3$, not all cospherical or coplanar, no four circular and no three collinear.
If $n\geqslant 883$, then the number of spheres determined by $S$ is at least $1+\binom{n-1}{3}-t_3^{orchard}(n-1)$. 
This bound is best possible.
\end{thm}

We begin with the following lemma.
\begin{lem}
\label{lem-n-1-k-coplanar}
Let $S$ be a set of $n$ points in $\mathbb{R}^3$, at most $n-k$ cospherical or coplanar, no four cocircular and no three collinear.
If $S' = Inv_p(S\backslash\{p\})$ for some $p\in S$, then $S'$ is a set of $n-1$ point with at most $n-1-k$ coplanar.
\end{lem}

Since spherical inversion is self-inverse, it's also the case that $S = Inv_p(S') \cup \{p\}$.
\begin{proof}
Let $\pi$ be a plane incident to exactly $r$ points in $S'$. There are two cases.
\begin{itemize}
\item Case: $p$ is incident to $\pi$.

The $Inv_p(\pi)$ is a plane containing $r+1$ points (including $p$) in $S$.
Since $r+1$ is at most $n-k$, the lemma follows.

\item Case: $p$ is not incident to $\pi$.

The $Inv_p(\pi)$ is a sphere containing $r+1$ points (including $p$) in $S$.
Since $r+1$ is at most $n-k$, the lemma follows.

\end{itemize}
\end{proof}

We now provide a lemma, analogous to the previous one, for when $n-k$ points are coplanar.
\begin{lem}
\label{lem-coplanar-1-off}
Let $S$ be a set of $n$ points in $\mathbb{R}^3$, exactly $n-k$ coplanar, no four cocircular and no three collinear.
Let $\pi$ be the plane to which $n-k$ points are incident.
Let $p\in S$ be any point not on $\pi$.
There are at least
$\binom{n-k}{3}$ spheres determined by $p$ and three points on $\pi$.  
Furthermore, each of these spheres is incident to at most $3+k$ points of $S$.
\end{lem}

\begin{proof}
Since no three points are collinear and no four cocircular, each point triple from the plane will determine a unique sphere with $p$. 
\end{proof}

The following two lemmas place an upper bound on the intersection of two sets.  
The proof for the primary theorem of this section utilizes the inclusion-exclusion principle, 
which requires an upper bound for the size of the intersection of two such sets.
\begin{lem}
\label{lem-cospherical-intersection}
Let $S$ be a set of $n$ points in $\mathbb{R}^3$, exactly $n-k$ ($k\geqslant 2$) cospherical, no four cocircular and no three collinear.
Let $\sigma_0$ be the sphere to which $n-k$ points are incident.
Let $p\in S$ and $q\in S$ be any two distinct points not on $\sigma_0$.
Let $\sigma_p$ be the set of spheres determined by $p$ and three points from $\sigma_0$.
Likewise, let $\sigma_q$ the set of spheres determined by $q$ and three points from $\sigma_0$.
It must be the case that $|\sigma_p\cap\sigma_q|\leqslant\frac{1}{3}\binom{n-k}{2}$.
\end{lem}
\begin{proof}
The spheres in $\sigma_p\cap\sigma_q$ all contain at least five and at most $3+k$ points.
Furthermore, no two distinct spheres in $\sigma_p\cap\sigma_q$ 
can share a pair of points from $\sigma_0$, since that pair 
along with $p$ and $q$ determine a sphere.
This reduces our problem to determining how many distinct point triples 
from $\sigma_0$ can there be such that no two triples share a pair.
From this, we see that $|\sigma_p\cap \sigma_q| \leqslant \frac{1}{3}\binom{n-k}{2}$.
\end{proof}

\begin{lem}
\label{lem-coplanar-intersection}
Let $S$ be a set of $n$ points in $\mathbb{R}^3$, exactly $n-k$ ($k\geqslant 2$) coplanar, no four cocircular and no three collinear.
Let $\pi$ be the plane to which $n-k$ points are incident.
Let $p\in S$ and $q\in S$ be any two distinct points not on $\pi$.
Let $\sigma_p$ be the set of spheres determined by $p$ and three points from $\pi$.
Likewise, let $\sigma_q$ the set of spheres determined by $q$ and three points from $\pi$.
It must be the case that $|\sigma_p\cap\sigma_q|\leqslant\frac{1}{3}\binom{n-k}{2}$.
\end{lem}
\begin{proof}
The spheres in $\sigma_p\cap\sigma_q$ all contain at least five and at most $3+k$ points.
Furthermore, no two distinct spheres in $\sigma_p\cap\sigma_q$ 
can share a pair of points from $\sigma_0$, since that pair 
along with $p$ and $q$ determine a sphere.
The lemma follows.
\end{proof}

\begin{proof}[Proof of Theorem \ref{thm-sph-5-pt}]

We shall consider all possible cases.

\begin{itemize}

\item Case: Exactly $n-1$ points in $S$ are cospherical.

Follows from Lemma \ref{lem-cospherical-1-off}.

\item Case: Exactly $n-1$ points in $S$ are coplanar.

Follows from Lemma \ref{lem-coplanar-1-off}.

\item Case: Exactly $n-2$ points in $S$ are cospherical.

Let $\sigma_0$ be the sphere incident to $n-2$ points.
Let $p$ and $q$ be the two points not on $\sigma_0$.
Let $\sigma_p$ be the set of spheres determined by $p$ and three points on $\sigma_0$, 
and let $\sigma_q$ be defined in similarly.

From the Inclusion-Exclusion Principle, we know $|\sigma_p\cup \sigma_q| = |\sigma_p| + |\sigma_q| - |\sigma_p\cap \sigma_q|$.
The spheres in $\sigma_p\cup\sigma_q$ are each incident to at most five points.

From Lemmas \ref{lem-cospherical-1-off} and \ref{lem-cospherical-intersection}, we see that
\begin{multline*}
|\sigma_p\cup\sigma_q| \geqslant 2\left\{ \binom{n-2}{3} - \frac{1}{3}\binom{n-2}{2} \right\}- \frac{1}{3}\binom{n-2}{2}
\\ \geqslant 2\binom{n-2}{3}-\binom{n-2}{2} \geqslant 1+\binom{n-1}{3}- t_3^{orchard}(n-1)
\end{multline*}
(for all $n\geqslant10$). Note that by the F\"{u}redi-Pal\'{a}sti configuration \cite{FP84},  
$t_3^{orchard}(n)\geqslant\lfloor\frac{1}{6}n^2-\frac{1}{2}n+1\rfloor$.

\item Case: Exactly $n-2$ points in $S$ are coplanar.

Let $\pi$ be the plane incident to $n-2$ points.
Let $p$ and $q$ be the two points not on $\pi$.
Let $\sigma_p$ be the set of spheres determined by $p$ and three points from $\pi$,
and let $\sigma_q$ be defined similarly.

From the Inclusion-Exclusion Principle, we know $|\sigma_p\cup \sigma_q| = |\sigma_p| + |\sigma_q| - |\sigma_p\cap \sigma_q|$.
The spheres in $\sigma_p\cup\sigma_q$ are each incident to at most five points.

From Lemmas \ref{lem-coplanar-1-off} and \ref{lem-coplanar-intersection}, we see that
\begin{multline*}
|\sigma_p\cup\sigma_q| \geqslant 2\binom{n-2}{3}- \frac{1}{3}\binom{n-2}{2}
\geqslant 1+\binom{n-1}{3}- t_3^{orchard}(n-1)
\end{multline*}
(for all $n\geqslant8$).

\item Case: Exactly $n-3$ points in $S$ are cospherical.

Let $\sigma_0$ be the sphere incident to $n-3$ points.
Let $p$, $q$ and $r$ be the three points not on $\sigma_0$.
Let $\sigma_p$ be the set of spheres determined by $p$ and three points on $\sigma_0$, 
and let $\sigma_q$ and $\sigma_r$ be defined similarly.

Again, the Inclusion-Exclusion Principle shows that
\[
|\sigma_p \cup \sigma_q \cup \sigma_r| \geqslant |\sigma_p| + |\sigma_q| + |\sigma_r| - |\sigma_p\cap \sigma_q| - |\sigma_p \cap \sigma_r| - |\sigma_q\cap \sigma_r|.
\]

The set $\sigma_p \cap \sigma_q \cap \sigma_r$ contains precisely those spheres incident to six points.
Since these spheres are already subtracted, we need not alter the formula above.

Therefore, the number of spheres incident to at most five point is at least
\begin{multline*}
3\left\{\binom{n-3}{3}- \frac{1}{3}\binom{n-3}{2}\right\}- \binom{n-3}{2}
\\= 3\binom{n-3}{3}-2\binom{n-3}{2}\geqslant  1+\binom{n-1}{3}- t_3^{orchard}(n-1)
\end{multline*}
(for all $n\geqslant 11$).

\item Case: Exactly $n-3$ points in $S$ are coplanar.

Let $\pi$ be the sphere incident to $n-3$ points.
Let $p$, $q$ and $r$ be the three points not on $\pi$.
Let $\sigma_p$ be the set of spheres determined by $p$ and three points on $\pi$, 
and let $\sigma_q$ and $\sigma_r$ be defined similarly.

Again, the Inclusion-Exclusion Principle shows that
\[
|\sigma_p \cup \sigma_q \cup \sigma_r| \geqslant |\sigma_p| + |\sigma_q| + |\sigma_r| - |\sigma_p\cap \sigma_q| - |\sigma_p \cap \sigma_r| - |\sigma_q\cap \sigma_r|.
\]

The set $\sigma_p \cap \sigma_q \cap \sigma_r$ contains precisely those spheres incident to six points.
Since these spheres are already subtracted, we need not alter the formula above.

Therefore, the number of spheres incident to at most five point is at least
\[
3\binom{n-3}{3}- \binom{n-3}{2} \geqslant  1+\binom{n-1}{3}- t_3^{orchard}(n-1)
\]
(for all $n\geqslant 10$).

\item Case: At most $n-4$ points in $S$ lie on any plane or sphere.

Let $S'=Inv_p(S\backslash\{p\}$ for some $p\in S$. 
$S'$ is a set of $n-1$ points with at most $n-3$ on any plane.

By combining Theorem \ref{thm-total-planes}, Corollary \ref{cor:lb-34-pt-planes} and Lemma \ref{lem-n-1-k-coplanar}, we know that if $n \geqslant 883$
then the number of planes determined at most four points in $S'$ is at least 
$\frac{5}{8}\left\{1+4\binom{n-5}{2}-\frac{n-5}{2}\binom{4}{2}\right\}=\frac{1}{8}(10 n^2-125 n+380)$.
We need an upper bound on how many of these planes pass through $p$.

From the proof of Lemma \ref{lem-cospherical-1-off}, we know that the number of planes determined 
by points in $S'$ passing through $p$
is less than $ \frac{1}{3}\binom{n-1}{2}$.
All other planes determined in $S'$ correspond to a sphere determined by points in $S$.
The number of spheres, passing through $p$, determined by at most five points in $S$ is at least
\[
\frac{1}{8}(10 n^2-125 n+380)- \frac{1}{3}\binom{n-1}{2}
= \frac{1}{24}(26 n^2-363 n+1132).
\] 

We can repeat this argument for all $n$ points in $S$, counting each sphere at most five times.
Thus, the number of spheres determined by $S$ is at least
\begin{multline*}
\frac{n}{5}\left(\frac{1}{24}\right)(26 n^2-363 n+1132)=\frac{1}{120}( 26 n^3-363 n^2+1132 n)
\\ \geqslant 1+\binom{n-1}{3}- t_3^{orchard}(n-1)
\end{multline*}
(for all $n\geqslant34$).
\end{itemize}

\end{proof}

\section{Conclusion}

In this paper, our results for point sets in $\mathbb{R}^3$ assumed no three points are collinear.
It would be interesting to know whether similar results could be obtained by assuming that
at most $j$ points are collinear, for some fixed $j\geqslant3$.

If one were to allow an arbitrary number of collinear points, the lower bounds for 
the number of objects determined (e.g., lines, planes, etc.) become weaker.
Consider the class of configurations which have exactly $n-k$ points collinear for a fixed $k$.
We note the following:
\begin{itemize}

\item At most $\binom{k}{2} + k(n-k) + 1 = O(n)$ lines are determined.

\item At most $\binom{k}{3} + (n-k)\binom{k}{2} + k = O(n)$ planes are determined.

\item At most $\binom{k}{4} +  (n-k)\binom{k}{3} +  \binom{k}{2}\binom{n-k}{2} = O(n^2)$ spheres are determined.

\end{itemize}

It would also be interesting to know whether the results of this paper
could be matched, or improved upon, using oriented matroids.

\bibliographystyle{IEEEtran}
\bibliography{planesSpheres}

\end{document}